\documentclass[12pt,sn-mathphys-ay]{sn-jnl}

\usepackage{graphicx}%
\usepackage{multirow}%
\usepackage{amsmath,amssymb,amsfonts}%
\usepackage{amsthm}%
\usepackage{mathrsfs}%
\usepackage[title]{appendix}%
\usepackage{xcolor}%
\usepackage{textcomp}%
\usepackage{manyfoot}%
\usepackage{booktabs}%
\usepackage{algorithm}%
\usepackage{algorithmicx}%
\usepackage{algpseudocode}%
\usepackage{listings}%
\usepackage{hyperref}
\usepackage{dsfont}
\theoremstyle{thmstyleone}%
\newtheorem{theorem}{Theorem}
\newtheorem{lemma}[theorem]{Lemma}%

\theoremstyle{thmstyletwo}%
\newtheorem{remark}{Remark}%

\theoremstyle{thmstylethree}%
\begin{document}
\newcommand{\R}{\mathds{R}}
\newcommand{\M}{M^{+}(\R^n)}
\newcommand{\wa}{\mathbf{W}_{\alpha,p}}
\newcommand{\ia}{\mathbf{I}_{2\alpha}}

\title{Radially symmetric solutions to a Lane-Emden type system}
\author{\fnm{Genival} \sur{da Silva}\footnote{email: gdasilva@tamusa.edu, website: \url{www.gdasilvajr.com}}}
\affil{\orgdiv{Department of Mathematics}, \orgname{Texas A\&M University - San Antonio}}


\abstract{The existence of radially symmetric solutions is discussed for a Lane-Emden type system. This answer a question posed by \cite{bib1}. We also comment on the inhomogeneous version of the same system and discuss some open questions.}




\maketitle
\section{Introduction}\label{sec1}
In this manuscript we analyze the existence of solutions to the system
\begin{equation}\label{FracLap}
     \left\{
\begin{aligned}
& (-\Delta)^{\alpha} u= \sigma\, v^{q_1}, \quad v>0 \quad \mbox{in}\quad \mathbb{R}^n,\\ 
& (-\Delta)^{\alpha} v= \sigma\, u^{q_2}, \quad u>0 \quad \mbox{in}\quad \mathbb{R}^n,\\
& \liminf\limits_{|x|\to \infty}u(x)=0, \quad \liminf\limits_{|x|\to \infty}v(x)=0.
\end{aligned}
\right.
 \end{equation}
 when $\sigma\in M^{+}(\mathds{R}^{n})$ is radially symmetric satisfying 
 $\mathbf{I}_{2\alpha}\sigma(x)\not\equiv\infty$ for almost every $x\in \mathds{R}^n$, or equivalently, 
\begin{equation}\label{FinPot}
    \int_1^\infty\left(\frac{\sigma(B(0,t))}{t^{n-2\alpha }}\right)\frac{\mathrm{d} t}{t}<\infty,
\end{equation}
 with $0<\alpha<n/2$ and $q_1,~ q_2\in(0,1)$.
 
 A solution $(u,v)$ to \eqref{FracLap} is understood in the sense
 \begin{equation}\label{IntSys}
 \left\{
\begin{aligned}
& u(x)=\mathbf{I}_{2\alpha}(v^{q_1}\mathrm{d} \sigma)(x), \quad x\in \mathbb{R}^n,\\ 
& v(x)=\mathbf{I}_{2\alpha}(u^{q_2}\mathrm{d} \sigma)(x), \quad x\in \mathbb{R}^n,
\end{aligned}
\right.
 \end{equation}
where $\mathbf{I}_{2\alpha}$ is the Riesz Potential defined by 
\begin{equation}\label{RieszPotential}
    \mathbf{I}_{\alpha}\sigma(x)=\int_{\mathbb{R}^n}\frac{\mathrm{d}\sigma(y)}{|x-y|^{n-\alpha}}, \quad x\in \mathbb{R}^n.
\end{equation}
This type of problem gained attention after the publication of the seminal paper \cite{brezis92}, where the authors gave necessary and sufficient conditions for existence and uniqueness of solutions of the single semilinear elliptic problem 
$$-\Delta u = \sigma(x) u^q$$ in $\mathds{R}^n (n \geq 3)$  with $0<q<1, \; 0 \not \equiv  \sigma \geq 0$ and $\sigma \in L^{\infty}_{\mathrm{loc}}(\mathds{R})$. They proved that a bounded solution exists if and and only if $\mathbf{I}_{2}\sigma(x)\in L^{\infty}(\mathds{R})$. Additionally, the following pointwise bound holds
\begin{equation}\label{BK estimate}
    c^{-1}\left(\mathbf{I}_2\sigma\right)^{\frac{1}{1-q}}\leq u\leq c\,(\mathbf{I}_2\sigma)
\end{equation}
where $c>0$ is a constant independent of $u$.

Other works followed, and interesting results were published generalizing \cite{brezis92} to a wider class of problems. For example, in \cite{maly1,maly2} the authors generalize the Brezis-Kamin estimates to the quasilinear problem $$-\Delta_{p} u=\sigma,$$ obtaining the estimates
\begin{equation}
    c^{-1}\mathbf{W}_{1,p} \sigma(x) \leq u(x) \leq c\,\mathbf{W}_{1,p}\sigma(x), \quad x\in \mathds{R}^n,
\end{equation}
where $\mathbf{W}_{\alpha,p}\sigma(x)$ is the Wolff potential of $\sigma$, introduced in \cite{wolff83} and defined by
\begin{equation}\label{wolff potential}
    \mathbf{W}_{\alpha,p}\sigma(x)=\int_0^{\infty} \left(\frac{\sigma(B(x,t))}{t^{n-\alpha p}}\right)^{\frac{1}{p-1}}\frac{\mathrm{d} t}{t}, \quad x\in\mathds{R}^n,
\end{equation}
for $\sigma\in {M}^+(\mathds{R}^n), \; 1<p<\infty, \; 0<\alpha<{n}/{p}$ and $B(x,t)$ is the open ball of radius $t>0$ centered at $x$. See the wonderful textbooks \cite{adams96,hei2006} for more on nonlinear potentials.
\begin{remark}
Notice that 
\begin{equation*}
    \mathbf{I}_{2\alpha}\sigma (x)=\int_{\mathds{R}^n} \frac{\mathrm{d}\sigma(y)}{|x-y|^{n-2\alpha}}=(n-2\alpha)\int_{0}^\infty \frac{\sigma(B(x,t))}{t^{n-2\alpha}}\frac{\mathrm{d} t}{t}=(n-2\alpha)\mathbf{W}_{\alpha,2}\sigma(x),
\end{equation*}
so the Wolff potential coincides (up to a constant) with the Riesz potential when $p=2$.
\end{remark}

In \cite{verb16} the authors analyze the equivalent Brezis-Kamin problem for the p-Laplacian $$-\Delta_p u=\sigma \, u^q \quad \text{in} \, \, \,  \mathds{R}^n,$$and obtain (not necessary bounded) solutions satisfying 
\begin{equation}\label{Verb17 estimate}
    c^{-1}\left(\mathbf{W}_{1,p}\sigma\right)^{\frac{p-1}{p-1-q}}\leq u\leq c\,\left( \mathbf{W}_{1,p}\sigma +\left(\mathbf{W}_{1,p}\sigma\right)^{\frac{p-1}{p-1-q}} \right)
\end{equation}
In \cite{bib1}, the authors generalize these ideas to systems of the form (\ref{FracLap}) and obtain similar estimates. More precisely they proved that under certain conditions there exists a solution pair $(u,v)$ to (\ref{FracLap}) satisfying the following conditions
\begin{equation}\label{estimateDoO}
     \begin{aligned}
    & c^{-1}\left(\mathbf{I}_{2\alpha}\sigma\right)^{\gamma_1}\leq u\leq c\left(\mathbf{I}_{2\alpha}\sigma + \left(\mathbf{I}_{2\alpha}\sigma\right)^{\gamma_1} \right),\\
   & c^{-1}\left(\mathbf{I}_{2\alpha}\sigma\right)^{\gamma_2}\leq v\leq c\left(\mathbf{I}_{2\alpha}\sigma + \left(\mathbf{I}_{2\alpha}\sigma\right)^{\gamma_2} \right),
    \end{aligned}
\end{equation}
where $$\gamma_1=\frac{1+q_1}{1-q_1q_2},\;\gamma_2=\frac{1+q_2}{1-q_1q_2}.$$
\begin{remark}\label{lowB}
Based on the assumption \eqref{FinPot}, they also show that all nontrivial solutions to \eqref{FracLap} satisfy the lower bounds in \eqref{estimateDoO}. 
\end{remark}
In the same paper they proposed a question of whether or not a criteria could be found in the case $\sigma$ is radially symmetric, generalizing  \cite[Prop. 5.2]{verb16} to the case of systems. Our first goal is to answer this question positively and prove the following theorem
\begin{theorem}\label{t1} Let $ \sigma \in M^+(\mathds{R}^n)$ be radially symmetric. Then there is a solution $(u,v)$ to system \eqref{FracLap} satisfying \eqref{estimateDoO}  if and only if  
\begin{equation} \label{limc}
\int_{|y| \geq 1}\frac{d\sigma(y)}{|y|^{n-2\alpha} } < \infty\text{ and }\;
\limsup\limits_{|x| \to 0} \frac{\frac{1}{|x|^{(n-2\alpha)\frac{1}{\gamma_{i}}}}\int_{B(0,|x|)} \frac{d\sigma(y)}{|y|^{r_{i}(n-2\alpha)}} }{ \int_{|y| \geq |x|} \frac{d\sigma(y)}{|y|^{n-2\alpha} }} < \infty, i=1,2.
\end{equation}
where $r_{1}=1-\frac{1}{\gamma_1}$ and $r_{2}=1-\frac{1}{\gamma_2}.$
\end{theorem}
Our second result consists of analyzing the existence of solutions to a generalization of system \eqref{FracLap}:
\begin{equation}\label{inFracLap}
     \left\{
\begin{aligned}
& (-\Delta)^{\alpha} u= \sigma\, v^{q_1}+ \mu_{1}\\
& (-\Delta)^{\alpha} v= \sigma\, u^{q_2},+\mu_{2}\\
& \liminf\limits_{|x|\to \infty}u(x)=0, \quad \liminf\limits_{|x|\to \infty}v(x)=0.
\end{aligned}
\right.
 \end{equation}
where $\sigma,\mu_{1},\mu_{2}\in {M}^+$ are not necessarily radially symmetric  but $\sigma$ satisfies the condition
\begin{equation}\label{sigcon}
           \sigma(E)\leq C_{\sigma}\,\mathrm{cap}_{\alpha,2}(E) \quad \mbox{for all compact sets } E\subset \mathds{R}^n.
\end{equation}
We understand system \eqref{inFracLap} as  
 \begin{equation}
    \left\{
\begin{aligned}
& u= \ia\left(v^{q_1}\mathrm{d} \sigma + d\mu_1\right), \\
& v= \ia\left(u^{q_2}\mathrm{d} \sigma+ d\mu_2\right),
\end{aligned}
\right.
\end{equation}
We'll show that if $\ia \mu_{i}(x) \leq \ia\sigma (x)$ then we still have existence, more precisely we have:
\begin{theorem}\label{t2} Let $ \sigma,\mu_{1},\mu_{2}\in M^+(\mathds{R}^n)$ not necessarily radially symmetric where $\sigma$ satisfies \eqref{sigcon}. Suppose that for $i=1,2$, there is a constant $C>0$ such that  $$\ia\mu_{i}(x)\leq C\, \ia\sigma(x)$$
Then there is a solution $(u,v)$ to system \eqref{inFracLap} satisfying \eqref{estimateDoO}.
\end{theorem}
\subsection*{Notations and definitions}  
We assume $\Omega\subseteq \R^n$ is a domain. We denote by $M^+(\Omega)$ the space of all nonnegative locally finite Borel measures on $\Omega$ and $\sigma(E)= \int_E d \sigma$ the $\sigma$-measure of a measurable set $E\subseteq \Omega$. The letter $c$ will always denote a positive constant which may vary from line to line.
\subsection*{Organization of the paper} 
In Sect.~\ref{condit}, we present a criteria for existence of solutions to \eqref{FracLap} when $\sigma$ is radial, in Sect.~\ref{main_proof} we give the proof of theorem \ref{t1} and in Sect.~\ref{pt2} we prove theorem \ref{t2}.

\section{Conditions for existence}\label{condit}

In this section we give necessary and sufficient conditions for existence of solutions to the system (\ref{FracLap}) in case $ \sigma \in M^+(\mathds{R}^n)$ is radially symmetric. 

The following result can be considered as a generalization of \cite[Prop. 5.1]{verb16}.
\begin{theorem} \label{con1} Let $\sigma \in M^+(\mathds{R}^n)$ be radially symmetric. Then there exists a nontrivial solution pair $(u,v)$ to \eqref{FracLap} if and only if 
\begin{equation} \label{radialcond}
\int_{|y|<1} \frac{d\sigma(y)}{|y|^{(n-2\alpha)r_{1}}} < \infty, \int_{|y|<1} \frac{d\sigma(y)}{|y|^{(n-2\alpha)r_{2}}} < \infty \text { and } \int_{|y|\geq 1} \frac{d\sigma(y)}{|y|^{n-2\alpha}}  < \infty,
\end{equation}
where $r_{i}=1-\frac{1}{\gamma_{i}}.$ Moreover, we have
\begin{equation}\label{uvbound}
     \begin{aligned}
    & u(x)\approx \left (\frac{1}{
|x|^{n-2\alpha}}\left(\int_{|y|<|x|} \frac{d\sigma(y)}{|y|^{(n-2\alpha)r_{1}}}\right)^{\gamma_{1}} + \left(\int_{|y|\geq|x|} \frac{d\sigma(y)}{|y|^{n-2\alpha}}\right)^{\gamma_{1}} \right )\\
   & v(y)\approx   \left (\frac{1}{
|y|^{n-2\alpha}}\left(\int_{|z|<|y|} \frac{d\sigma(z)}{|z|^{(n-2\alpha)r_{2}}}\right)^{\gamma_{2}} + \left(\int_{|z|\geq|y|} \frac{d\sigma(z)}{|z|^{n-2\alpha}}\right)^{\gamma_{2}} \right ).
        \end{aligned}
\end{equation}
\end{theorem}
\begin{proof}~
Recall that if $\sigma$ is radial then $$\ia\sigma(x)\approx  \frac{\sigma(B(0,|x|)}{|x|^{n-2\alpha}} + \int_{|y|\geq|x|} \frac{d\sigma(y)}{|y|^{n-2\alpha} } ,$$ where we omit the first term when $x=0$. Moreover, we can assume $u,v$ are radial since $\ia\sigma(x)$ is radial in this case.\begin{enumerate}
    \item[($\Rightarrow$)] Suppose  $(u,v)$ is a solution to \eqref{FracLap}. Then according to remark \ref{lowB}
    \begin{equation}\label{ine1}
     \begin{aligned}
    & c\left(\mathbf{I}_{2\alpha}\sigma\right)^{\gamma_1}\leq u\\
   & c\left(\mathbf{I}_{2\alpha}\sigma\right)^{\gamma_2}\leq v.
        \end{aligned}
    \end{equation}
     In particular, we have
     \begin{equation}
    c\left(\int_{|y|\geq x} \frac{d\sigma(y)}{|y|^{n-2\alpha}}\right)^{\gamma_1}\leq u,
    \end{equation}
    but since $u$ is finite, this implies $\int_{|y|\geq 1} \frac{d\sigma(y)}{|y|^{n-2\alpha}}  < \infty$. Likewise,
\begin{equation}
    c\left(\frac{1}{|x|^{(n-2\alpha)}}\int_{|y|\leq x} \frac{|y|^{(n-2\alpha)r_{1}}d\sigma(y)}{|y|^{(n-2\alpha)r_{1}}}\right)^{\gamma_1}\leq u,
    \end{equation}
    For any $\delta>0$ such that $\delta<|y|$ we have
    \begin{equation}
    c\left(\frac{\delta^{(n-2\alpha)r_{1}}}{|x|^{(n-2\alpha)}}\int_{\delta<|y|\leq x} \frac{d\sigma(y)}{|y|^{(n-2\alpha)r_{1}}}\right)^{\gamma_1}\leq\left(\frac{1}{|x|^{(n-2\alpha)}}\int_{|y|\leq x} \frac{|y|^{(n-2\alpha)r_{1}}d\sigma(y)}{|y|^{(n-2\alpha)r_{1}}}\right)^{\gamma_1}\leq u,
    \end{equation}
    By symmetry, the exact same argument works with $v, r_{2}, \gamma_{2}$ instead of $u,r_{1},\gamma_{1}$.
  \item[($\Leftarrow$)]  Conversely, suppose \eqref{radialcond} holds. 	
Notice that 
\begin{equation*}
     \gamma_1=q_{1}\gamma_{2}+1,\;\gamma_2=q_{2}\gamma_{1}+1,\;
\end{equation*}
According to \cite[Proof of thm. 1]{bib1}, there is a $\lambda_{1}>0$ sufficiently small such that
\begin{equation*}
    (\underline{u},\underline{v})= \big(\lambda_1 (\ia\sigma)^{\gamma_1}, \lambda_1 (\ia\sigma)^{{\gamma_2}}\big)
\end{equation*}
is a subsolution to \eqref{FracLap}. So it's enough to find supersolution $(\overline{u},\overline{v})$ such that $\overline{u}\geq \underline{u}$ and $\overline{v}\geq \underline{v}$. Set
\begin{equation*}
     \begin{aligned}
    & \overline{u}(x)=\lambda \, \left (\frac{1}{
|x|^{n-2\alpha}}\left(\int_{|y|<|x|} \frac{d\sigma(y)}{|y|^{(n-2\alpha)r_{1}}}\right)^{\gamma_{1}} + \left(\int_{|y|\geq|x|} \frac{d\sigma(y)}{|y|^{n-2\alpha}}\right)^{\gamma_{1}} \right )\\
   & \overline{v}(y)= \lambda \, \left (\frac{1}{
|y|^{n-2\alpha}}\left(\int_{|z|<|y|} \frac{d\sigma(z)}{|z|^{(n-2\alpha)r_{2}}}\right)^{\gamma_{2}} + \left(\int_{|z|\geq|y|} \frac{d\sigma(z)}{|z|^{n-2\alpha}}\right)^{\gamma_{2}} \right )
        \end{aligned}
\end{equation*}
    where $\lambda$ is a constant to be determined later.
    
We claim that $ \overline{u} \geq \ia( \overline{v}^{q_{1}}d\sigma)$. For $x \neq 0$,  we have 
\begin{equation*}
\begin{aligned} 
\ia( \overline{v}^{q_{1}}d\sigma) & \approx \frac{1}{|x|^{n-2\alpha}}\int_{|y|<|x|}\overline{v}^{q_{1}}d\sigma+\,\int_{|y| \ge|x|}\frac{\overline{v}^{q_{1}}d\sigma}{|y|^{n-2\alpha}}\\
& \leq \lambda^{q_{1}}\,\frac{1}{|x|^{n-2\alpha}}\int_{|y|<|x|}\frac{1}{|y|^{(n-2\alpha)q_{1}}}\left(\int_{|z|<|y|} \frac{d\sigma(z)}{|z|^{(n-2\alpha)r_{2}}}\right)^{q_{1}\gamma_{2}}d\sigma(y)\\
&+ \lambda^{q_{1}}\,\frac{1}{|x|^{n-2\alpha}}\int_{|y|<|x|}\left(\int_{|z|\ge |y|} \frac{d\sigma(z)}{|z|^{n-2\alpha}}\right)^{q_{1}\gamma_{2}}d\sigma(y)\\
&+\lambda^{q_{1}}\,\int_{|y| \ge|x|}\frac{1}{|y|^{n-2\alpha}}\frac{1}{|y|^{(n-2\alpha)q_{1}}}\left(\int_{|z|<|y|} \frac{d\sigma(z)}{|z|^{(n-2\alpha)r_{2}}}\right)^{q_{1}\gamma_{2}}d\sigma(y)\\
&+\lambda^{q_{1}}\,\int_{|y| \ge|x|}\frac{1}{|y|^{n-2\alpha}}\left(\int_{|z|\ge |y|} \frac{d\sigma(z)}{|z|^{n-2\alpha}}\right)^{q_{1}\gamma_{2}}d\sigma(y)\\&:=\lambda^{q_{1}}(I+II+III+IV).
\end{aligned}
\end{equation*} 
Notice that
\begin{equation*} 
\begin{aligned} 
I&= \,\frac{1}{|x|^{n-2\alpha}}\int_{|y|<|x|}\frac{1}{|y|^{(n-2\alpha)q_{1}}}\left(\int_{|z|<|y|} \frac{d\sigma(z)}{|z|^{(n-2\alpha)r_{1}}|z|^{(n-2\alpha)(r_{2}-r_{1})}}\right)^{q_{1}\gamma_{2}}d\sigma(y)\\
&\leq \frac{1}{|x|^{n-2\alpha}}\int_{|y|<|x|}\frac{1}{|y|^{(n-2\alpha)q_{1}(1+\gamma_{2}(r_{2}-r_{1}))}}\left(\int_{|z|<|y|} \frac{d\sigma(z)}{|z|^{(n-2\alpha)r_{1}}}\right)^{q_{1}\gamma_{2}}d\sigma(y)\\
&\leq \frac{1}{|x|^{n-2\alpha}}\left(\int_{|y|<|x|} \frac{d\sigma(y)}{|y|^{(n-2\alpha)r_{1}}}\right)^{\gamma_{1}}.
\end{aligned} 
\end{equation*}

For the sake of convenience, we'll split $II$ in 2 parts:
\begin{eqnarray*}
II & = &  \frac{1}{|x|^{n-2\alpha}}\int_{|y|<|x|}\left(\int_{|y|\leq|z|<|x|} \frac{d\sigma(z)}{|z|^{n-2\alpha}} + \int_{|z|\geq|x|} \frac{d\sigma(z)}{|z|^{n-2\alpha}}\right)^{q_{1}\gamma_{2}}d\sigma(y)\\
& \le & c\frac{1}{|x|^{n-2\alpha}}\int_{|y|<|x|} \left(\int_{|y|\leq|z|<|x|} \frac{d\sigma(z)}{|z|^{n-2\alpha}}\right)^{q_{1}\gamma_{2}}d\sigma(y)\\ 
& + & c,\frac{1}{|x|^{n-2\alpha}}\int_{|y|<|x|} d\sigma(y)\left(\int_{|z|\ge|x|} \frac{d\sigma(z)}{|z|^{n-2\alpha}}\right)^{q_{1}\gamma_{2}} \\ 
&= & c(II_a+II_b).
\end{eqnarray*}
So
\begin{equation*}
\begin{aligned} 
II_a & \leq \frac{1}{|x|^{n-2\alpha}}\int_{|y|<|x|} \left(\int_{y\leq |z|<|x|} \frac{d\sigma(z)}{|z|^{(n-2\alpha)r_{1}}|z|^{(n-2\alpha)(1-r_{1})}}\right)^{q_{1}\gamma_{2}} d\sigma(y)\\
& \leq \frac{1}{|x|^{n-2\alpha}}\int_{|y|<|x|}  \frac{1}{|y|^{(n-2\alpha)(1-r_{1})q_{1}\gamma_{2}}}\left(\int_{y\leq |z|<|x|} \frac{d\sigma(z)}{|z|^{(n-2\alpha)r_{1}}}\right)^{q_{1}\gamma_{2}} d\sigma(y)\\
& \leq \frac{1}{|x|^{n-2\alpha}}\left(\int_{ |y|<|x|} \frac{d\sigma(y)}{|y|^{(n-2\alpha)r_{1}}}\right)^{\gamma_{1}} 
\end{aligned}
\end{equation*} 
and using Young's inequality with $\gamma_{1}$ and $\frac{1+q_{1}}{q_{1}+q_{1}q_{2}}$ we obtain we obtain
\begin{equation*}
\begin{aligned} 
II_b & =\frac{1}{|x|^{n-2\alpha}}\int_{|y|<|x|} d\sigma(y)\left(\int_{|z|\ge|x|} \frac{d\sigma(z)}{|z|^{n-2\alpha}}\right)^{q_{1}\gamma_{2}} \\
& \le c\, \left(\left(\frac{1}{|x|^{n-2\alpha}}\int_{|y|<|x|} d\sigma(y)\right)^{\gamma_{1}}+ \left(\int_{|z|\geq|x|} \frac{d\sigma(z)}{|z|^{n-2\alpha}}\right)^{\gamma_{1}} \right).
\end{aligned}
\end{equation*} 
Next, we  have
\begin{equation*}
\begin{aligned} 
&III \le c\,\int_{|y| \ge|x|}\frac{1}{|y|^{n-2\alpha}}\frac{1}{|y|^{(n-2\alpha)q_{1}}}\left(\int_{|z|<|x|} \frac{d\sigma(z)}{|z|^{(n-2\alpha)r_{2}}}\right)^{q_{1}\gamma_{2}}d\sigma(y)\\
&+c\,\int_{|y| \ge|x|}\frac{1}{|y|^{n-2\alpha}}\frac{1}{|y|^{(n-2\alpha)q_{1}}}\left(\int_{|x|\leq |z|\leq |y|} \frac{d\sigma(z)}{|z|^{(n-2\alpha)r_{2}}}\right)^{q_{1}\gamma_{2}}d\sigma(y)\\
&\le c\,\frac{1}{|x|^{(n-2\alpha)q_{1}}}\left(\int_{|z|<|x|}\frac{d\sigma(z)}{|z|^{(n-2\alpha)r_{2}}}\right)^{q_{1}\gamma_{2}}\int_{|y|\ge|x|} \frac{d\sigma(y)}{|y|^{n-2\alpha}}\\
&+c\,\int_{|y|\ge|x|} \frac{1}{|y|^{n-2\alpha}|y|^{(n-2\alpha)q_{1}}}\left(\int_{|x|\le |z|<|y|}\frac{|z|^{(n-2\alpha)(\frac{1}{\gamma_{2}})}d\sigma(z)}{|z|^{n-2\alpha}}\right)^{q_{1}\gamma_{2}}
d\sigma(y)\\
&\le c\,\frac{1}{|x|^{(n-2\alpha)q_{1}}}\left(\int_{|z|<|x|}\frac{d\sigma(z)}{|z|^{(n-2\alpha)r_{1}}|z|^{(n-2\alpha)(r_{2}-r_{1})}}\right)^{q_{1}\gamma_{2}}\int_{|y|\ge|x|} \frac{d\sigma(y)}{|y|^{n-2\alpha}}\\
&+c\,\left(\int_{|x|\le |z|}\frac{d\sigma(z)}{|z|^{n-2\alpha}}\right)^{q_{1}\gamma_{2}}\int_{|y|\ge|x|} \frac{d\sigma(y)}{|y|^{(n-2\alpha)}}\\
&\le c\,\frac{1}{|x|^{(n-2\alpha)q_{1}(1+\gamma_{2}(r_{2}-r_{1}))}}\left(\int_{|z|<|x|}\frac{d\sigma(z)}{|z|^{(n-2\alpha)r_{1}}}\right)^{q_{1}\gamma_{2}}\int_{|y|\ge|x|} \frac{d\sigma(y)}{|y|^{n-2\alpha}}\\
&+c\,\left(\int_{|y|\geq |x|}\frac{d\sigma(y)}{|y|^{n-2\alpha}}\right)^{\gamma_{1}}\\
&\le c \left(\,\frac{1}{|x|^{(n-2\alpha)}}\left(\int_{|z|<|x|}\frac{d\sigma(z)}{|z|^{(n-2\alpha)r_{1}}}\right)^{\gamma_{1}}+\left(\int_{|y|\ge|x|} \frac{d\sigma(y)}{|y|^{n-2\alpha}}\right)^{\gamma_{1}}\right)\\
&+c\,\left(\int_{|y|\geq |x|}\frac{d\sigma(y)}{|y|^{n-2\alpha}}\right)^{\gamma_{1}}
\end{aligned}
\end{equation*}
Where we applied Young's inequality in the last inequality. Finally, notice that 
\begin{equation*}
IV \leq \left(\int_{|y|\geq |x|}\frac{d\sigma(y)}{|y|^{n-2\alpha}}\right)^{\gamma_{1}}.
\end{equation*} 
By choosing $\lambda$ large enough we then guarantee that $\ia(\overline{v}^{q_{1}}d\sigma) \le \overline{u}$, and by symmetry it's also possible to conclude that $\ia(\overline{u}^{q_{2}}d\sigma) \le \overline{v}$. 

To obtain a solution $(u,v)$ we use the standard iteration argument of the sub-sup solution method of PDEs and the monotone convergence theorem. We reproduce the full argument for the convenience of the reader.

Let $u_0=\underline{u}$ and $v_0=\underline{v}$. Clearly, $u_0\leq \overline{u}$ and $v_0\leq \overline{v}$. Set $u_1=\ia(v_0^{q_1}\mathrm{d}\sigma)$ and $v_1=\ia(u_0^{q_2}\mathrm{d}\sigma)$. We have $u_1\geq u_0$ and $v_1\geq v_0$. We interate this process and obtain a sequence of pair of functions $(u_j,v_j)$  such that
\begin{equation}
\left\{\begin{aligned}
&u_j=\ia(v_{j-1}^{q_1}\mathrm{d} \sigma) \quad \mbox{in}\quad \mathds{R}^n,\\
& v_j=\ia(u_{j-1}^{q_1}\mathrm{d} \sigma)\quad \mbox{in}\quad \mathds{R}^n, 
\end{aligned}\right.    
\end{equation}
By induction, the sequences $\{u_j\}$ and $\{v_j\}$ are nondecreasing, with $\underline{u}\leq u_j\leq\overline{u}$ and $\underline{v}\leq v_j\leq\overline{v}$ (for $j=0,1,\ldots$). 

Using the  Monotone Convergence Theorem and taking the limit as $j\to \infty$, we see that there exist nonnegative functions $u=\lim u_j$ and $v=\lim v_j$ such that $(u,v)$ is a solution satisfying $\underline{u}\leq u\leq\overline{u}$ and $\underline{v}\leq v\leq\overline{v}$. In particular, $u$ and $v$ satisfies \eqref{uvbound}. 
\end{enumerate}
\end{proof}
\section{Proof of theorem 1}\label{main_proof}
Consider the following condition:
\begin{equation}\label{c114}
\begin{aligned}
&    \ia\left(\left(\ia\sigma\right)^{q_{1}\gamma_{2}} \mathrm{d}\sigma\right)\leq c \left(\ia\sigma + \left(\ia\sigma\right)^{\gamma_{1}} \right), \\
& \ia\left(\left(\ia\sigma\right)^{q_{2}\gamma_{1}}\mathrm{d} \sigma\right)\leq c \left(\ia\sigma + \left(\ia\sigma\right)^{\gamma_{2}} \right).
\end{aligned}
 \end{equation}
 According to \cite[Thm. 1.3]{bib1}, these conditions are necessary and sufficient for the existence of solutions satisfying a Brezis-Kamin type estimates \eqref{estimateDoO}. Set
 \begin{equation}
 \begin{aligned}
 \mathbf{K_{1}}\sigma(x) &= \frac{1}{|x|^{n-2\alpha}}\left(\int_{|y|<|x|}\frac{d\sigma(y)}{|y|^{(n-2\alpha)r_{1}}}\right)^{\gamma_{1}}, \quad x \neq 0\\
  \mathbf{K_{2}}\sigma(x) &= \frac{1}{|x|^{n-2\alpha}}\left(\int_{|y|<|x|}\frac{d\sigma(y)}{|y|^{(n-2\alpha)r_{2}}}\right)^{\gamma_{2}}, \quad x \neq 0
 \end{aligned}
 \end{equation}
 \begin{lemma}
 Condition \eqref{c114} is equivalent to 
  \begin{equation}\label{con2}
     \begin{aligned}
    & \mathbf{K_{1}}\sigma(x) \leq c\left(\mathbf{I}_{2\alpha}\sigma + \left(\mathbf{I}_{2\alpha}\sigma\right)^{\gamma_{1}} \right)<\infty,\;\; \text{a.e.}\\
   &  \mathbf{K_{2}}\sigma(x)\leq c\left(\mathbf{I}_{2\alpha}\sigma + \left(\mathbf{I}_{2\alpha}\sigma\right)^{\gamma_{2}} \right)<\infty,\;\; \text{a.e.}
    \end{aligned}
 \end{equation}
 \end{lemma}
 \begin{proof}
 Suppose \eqref{c114} holds. Then the there is a solution pair $(u,v)$ satisfying \eqref{estimateDoO}, moreover by theorem \ref{con1} they satisfy:
 \begin{equation}
     \begin{aligned}
    & \mathbf{K_{1}}\sigma(x) \leq u\leq c\left(\mathbf{I}_{2\alpha}\sigma + \left(\mathbf{I}_{2\alpha}\sigma\right)^{\gamma_{1}} \right),\\
   &  \mathbf{K_{2}}\sigma(x)\leq v\leq c\left(\mathbf{I}_{2\alpha}\sigma + \left(\mathbf{I}_{2\alpha}\sigma\right)^{\gamma_{2}} \right),
    \end{aligned}
\end{equation}
which implies \eqref{con2}.

Conversely, suppose \eqref{con2} holds. Using theorem \ref{con1} again, we can guarantee the existence of a solution $(u,v)$ satisfying:
\begin{equation}
     \begin{aligned}
    & u(x)\leq c\left (\mathbf{K_{1}}\sigma(x)+ \left(\ia\sigma\right)^{\gamma_{1}} \right )\leq c\left (\ia\sigma(x)+ \left(\ia\sigma\right)^{\gamma_{1}} \right )\\
   & v(y)\leq c  \left (\mathbf{K_{2}}\sigma(x) + \left(\ia\sigma\right)^{\gamma_{2}} \right ) \leq c\left (\ia\sigma(x)+ \left(\ia\sigma\right)^{\gamma_{2}} \right )
        \end{aligned}
\end{equation}
and by remark \ref{lowB}, we also have:
\begin{equation}
     \begin{aligned}
    &  u(x)= \ia (v^{q_{1}}d\sigma) \geq \ia\left(\left(\ia\sigma\right)^{q_{1}\gamma_{2}} \mathrm{d}\sigma\right)\\
   &  v(x)=\ia (u^{q_{2}}d\sigma) \geq \ia\left(\left(\ia\sigma\right)^{q_{2}\gamma_{1}} \mathrm{d}\sigma\right)\\
        \end{aligned}
\end{equation}
which completes the proof.
 \end{proof}
 We conclude the proof of the theorem with the following lemma:
 \begin{lemma}
 Condition \eqref{con2} is equivalent to
 \begin{equation} \label{con3}
\int_{|y| \geq 1}\frac{d\sigma(y)}{|y|^{n-2\alpha} } < \infty\text{ and }\;
\limsup\limits_{|x| \to 0} \frac{\frac{1}{|x|^{(n-2\alpha)\frac{1}{\gamma_{i}}}}\int_{B(0,|x|)} \frac{d\sigma(y)}{|y|^{r_{i}(n-2\alpha)}} }{ \int_{|y| \geq |x|} \frac{d\sigma(y)}{|y|^{n-2\alpha} }} < \infty, i=1,2.
\end{equation}
 \end{lemma}
 \begin{proof}
 Since $\lim_{|x| \to \infty} \ia\sigma(x) = 0$, it's enough to see what happens when $\limsup\limits_{|x| \to 0}  \ia\sigma(x)=\infty$, otherwise $\ia\sigma$ is bounded and \eqref{c114} holds trivially. 
 
 The proof is similar to the single equation case treated in \cite[Prop. 5.2]{verb16}. Suppose \eqref{con3} holds, then for x sufficiently small, say $0<|x|<\delta<1$ we have:
 $$\mathbf{K}_{i}\sigma(x) \le c \, (\ia\sigma(x))^{\gamma_{i}} \quad \text{for all}  \,\,  0 < |x| < \delta.$$
 If $0<\delta<|x|<1$ we have:
 $$\mathbf{K}_{i}\sigma(x) \leq \frac{1}{\delta^{n-2\alpha}} \left(\int_{|y|<1}\frac{d\sigma(y)}{|y|^{(n-2\alpha)r_{i}}}\right)^{\gamma_{i}}<c\;\ia\sigma(x)$$
 Now, for $|x|>1$:
 $$\int_{|y|<|x|}\frac{d\sigma(y)}{|y|^{(n-2\alpha)r_{i}}}\leq \int_{0\leq |y|\leq |1|}\frac{d\sigma(y)}{|y|^{(n-2\alpha)r_{i}}}+\int_{1\leq |y|\leq |x|}\frac{d\sigma(y)}{|y|^{(n-2\alpha)r_{i}}}$$
 The first term is clearly bounded by $\ia\sigma(x)$, we analyze the second term. Using Holder's inequality we obtain:
  $$\int_{1\leq |y|\leq |x|}\frac{d\sigma(y)}{|y|^{(n-2\alpha)r_{i}}}\leq\left( \int_{1\leq |y|\leq |x|}d\sigma(y)\right)^{1-\frac{1}{r_{i}}} \left(\int_{1\leq |y|\leq |x|}\frac{d\sigma(y)}{|y|^{(n-2\alpha)}}\right)^{\frac{1}{r_{i}}}$$
  Therefore,
  $$\left(\int_{1\leq |y|\leq |x|}\frac{d\sigma(y)}{|y|^{(n-2\alpha)r_{i}}}\right)^{\gamma_{i}} \leq\left( \int_{|y|\leq |x|}d\sigma(y)\right)^{\gamma_{i}(1-\frac{1}{r_{i}})} \left(\int_{1\leq |y|\leq |x|}\frac{d\sigma(y)}{|y|^{(n-2\alpha)}}\right)^{\frac{\gamma_{i}}{r_{i}}}$$
  and again we have:
  $$\mathbf{K}_{i}\sigma(x) \le c \, (\ia\sigma(x))$$
  Combining all cases together, we conclude that 
    $$\mathbf{K}_{i}\sigma(x) \le c \, \left(\ia\sigma(x)+(\ia\sigma(x))^{\gamma_{i}}\right)<\infty$$
   Now suppose \eqref{con2} holds. Since we are assuming $\limsup\limits_{|x| \to 0}  \ia\sigma(x)=\infty$, for $x$ sufficiently small we have
   $$\ia\sigma(x)\leq (\ia\sigma(x))^{\gamma_{i}},$$
   hence
   $$\mathbf{K}_{i}\sigma(x) \le c \, (\ia\sigma(x))^{\gamma_{i}},$$
   or equivalently,
   \begin{equation}\label{con5}
   \frac{1}{|x|^{(n-2\alpha)\frac{1}{\gamma_{i}}}}\left(\int_{|y|<|x|}\frac{d\sigma(y)}{|y|^{(n-2\alpha)r_{i}}}\right)\leq c\left(\frac{1}{|x|^{n-2\alpha}}
   \int_{|y|<|x|}d\sigma(y)  
+ \int_{|y|\geq |x|} \frac{d\sigma(y)}{|y|^{n-2\alpha} } \right)
\end{equation}
It suffices to estimate the first term. Let $d=f(|x|)$ be a function to be determined later, we have
\begin{equation}
     \begin{aligned}
     \frac{1}{|x|^{n-2\alpha}}\int_{|y|<|x|}d\sigma(y)&=\frac{1}{|x|^{(n-2\alpha)(1-r_{i})}|x|^{(n-2\alpha)r_{i}}}\int_{|y|<d}\frac{|y|^{(n-2\alpha)r_{i}}}{|y|^{(n-2\alpha)r_{i}}}d\sigma(y)\\
     &+ \frac{1}{|x|^{n-2\alpha}}\int_{d<|y|<x}\frac{|y|^{n-2\alpha}}{|y|^{n-2\alpha}}d\sigma(y)\\	
     &\leq \frac{d^{(n-2\alpha)r_{i}}}{|x|^{(n-2\alpha)(1-r_{i})}|x|^{(n-2\alpha)r_{i}}}\int_{|y|<d}\frac{1}{|y|^{(n-2\alpha)r_{i}}}d\sigma(y)\\
     &+\int_{d<|y|<x}\frac{d\sigma(y)}{|y|^{n-2\alpha}}\\	
     \end{aligned}
\end{equation}
Set $d=k|x|$, for $k$ small enough we have
\begin{equation}
     \begin{aligned}
     \frac{1}{|x|^{n-2\alpha}}\int_{|y|<|x|}d\sigma(y)&\leq \frac{k^{(n-2\alpha)r_{i}}}{|x|^{(n-2\alpha)(\frac{1}{\gamma_{i}})}}\int_{|y|<|x|}\frac{1}{|y|^{(n-2\alpha)r_{i}}}d\sigma(y)\\
     &+\int_{|x|\leq |y|}\frac{d\sigma(y)}{|y|^{n-2\alpha}}\\	
     \end{aligned}
\end{equation}
 Together with \eqref{con5}, this implies that 
 $$ \frac{\frac{1}{|x|^{(n-2\alpha)\frac{1}{\gamma_{i}}}}\int_{B(0,|x|)} \frac{d\sigma(y)}{|y|^{r_{i}(n-2\alpha)}} }{ \int_{|y| \geq |x|} \frac{d\sigma(y)}{|y|^{n-2\alpha} }} < c,\;\; i=1,2.$$
 for some constant $c$ and $x$ sufficiently small, which completes the proof.
 \end{proof}
 \begin{remark}
 As explained in \cite{verb16} in the case $u=v$ and $q_{1}=q_{2}=q$, it's possible to find $\sigma$ satisfying \ref{radialcond} but not \eqref{limc}. So we can have solutions that do not satisfy the Brezis-Kamin type estimates \eqref{estimateDoO}. 
 \end{remark}
  \section{Proof of Theorem 2}\label{pt2}
  
 In this section we analyze the system:
 \begin{equation}\label{inho}
    \left\{
\begin{aligned}
& u= \ia\left(v^{q_1}\mathrm{d} \sigma + d\mu_1\right), \\
& v= \ia\left(u^{q_2}\mathrm{d} \sigma+ d\mu_2\right),
\end{aligned}
\right.
\end{equation}
and give the proof of theorem \ref{t2}.

The proof is based on the proof of \cite[Thm 1]{bib1} \textit{mutatis mutandis} because of the condition $\ia\mu_{i}\leq C \ia\sigma$. Although not obvious, we'll see that this condition enables the existence of a supersolution.
\subsection*{A subsolution $(\underline{u},\underline{v})$}
Recall the following lemma from \cite{verb17} (simplified to our setting):
\begin{lemma}
    Let $\sigma \in M^+(\mathds{R}^n)$. For every $r>0$ and for all $x\in \mathds{R}^n$, it holds
    \begin{equation}
        \ia\left((\ia\sigma)^r\mathrm{d} \sigma\right)(x)\geq \kappa^{r}\left(\ia\omega(x)\right)^{r+1},
    \end{equation}
    where $\kappa$ depends only on $n, p$ and $\alpha$.
\end{lemma}
Consider the following pair of functions
\begin{equation*}
    (\underline{u},\underline{v})= \big(\lambda (\ia\sigma)^{\gamma_1}, \lambda (\ia\sigma)^{{\gamma_2}}\big),
\end{equation*}
where $\lambda$ is a constant to be found.  Then we have
\begin{equation*}
\begin{aligned}
    \ia\left(\underline{v}^{q_1}\mathrm{d} \sigma + d\mu_1\right)&= {\lambda}^{q_1}\ia( (\ia\sigma)^{q_1{\gamma_2}}\mathrm{d}\sigma + d\mu_1)\geq {\lambda}^{q_{1}}\left[\kappa^{q_{1}{\gamma_2}}(\ia\sigma)^{q_{1}{\gamma_2}+1} +\ia \mu_{1}\right]\\
    & \geq  {\lambda}^{q_{1}}\left[\kappa^{q_{1}{\gamma_2}}(\ia\sigma)^{{\gamma_1}} \right],\\    
     \ia\left(\underline{u}^{q_2}\mathrm{d} \sigma + d\mu_2\right)&= {\lambda}^{q_2}\ia( (\ia\sigma)^{q_2{\gamma_1}}\mathrm{d}\sigma + d\mu_2)\geq {\lambda}^{q_{2}}\left[\kappa^{q_{2}{\gamma_1}}(\ia\sigma)^{q_{2}{\gamma_1}+1} +\ia \mu_{2}\right]\\
    & \geq  {\lambda}^{q_{2}}\left[\kappa^{q_{2}{\gamma_1}}(\ia\sigma)^{{\gamma_2}} \right].\\   
\end{aligned}
\end{equation*}
Therefore, by choosing $\lambda$ sufficiently small we conclude that
\begin{equation*}
\begin{aligned}
    &\underline{u}\leq \ia(\underline{v}^{q_1}\mathrm{d}\sigma + d\mu_1),\\
    & \underline{v}\leq \ia(\underline{u}^{q_2}\mathrm{d}\sigma+ d\mu_2).
\end{aligned}
\end{equation*}
\subsection*{A supersolution $(\overline{u},\overline{v})$}
Now we show that there exists $\lambda>0$ sufficiently large such that
\begin{equation*}
(\overline{u},\overline{v})=\big(\lambda\left(\ia\sigma + \left(\ia\sigma\right)^{\gamma_1}\right), \lambda\left(\ia\sigma + \left(\ia\sigma\right)^{{\gamma_2}}\right)\big)
\end{equation*}
is a supersolution to \eqref{inho}, such that $\overline{u}\geq \underline{u}$ and $\overline{v}\geq \underline{v}$. The following lemma from \cite{verb17} (simplified to our setting) will be needed
\begin{lemma}\label{strcon}
       Let $\sigma\in {M}^+(\mathds{R}^n)$ satisfying \eqref{sigcon}. 
       Then, for every $s>0$, 
       \begin{equation}\label{con10}
           \int_E(\ia\sigma_E)^s\,\mathrm{d}\sigma\leq c\, \sigma(E) \quad \mbox{for all compact sets } E\subset \mathds{R}^n,
       \end{equation}
       where $c$ is a positive constant. Moreover, if \eqref{con10} holds for a given $s>0$, then \eqref{sigcon} holds; hence \eqref{con10} holds for every $s>0$.
\end{lemma}
We have
\begin{align}
    \ia(\overline{v}^{q_1}\mathrm{d}\sigma + d\mu_1)(x)&\leq {\lambda}^{q_{1}}\left[c\int_0^{\infty}\frac{\int_{B(x,t)}(\ia\sigma)^{q_1}\mathrm{d}\sigma +(\ia\sigma)^{{\gamma_2}q_1}\mathrm{d}\sigma }{t^{n-2\alpha }}\frac{\mathrm{d} t}{t} + \ia\mu_1(x)\right]\nonumber\\
    &\leq {\lambda}^{q_{1}}\left[c(\int_0^{\infty}\frac{\int_{B(x,t)}(\ia\sigma)^{q_1}\mathrm{d}\sigma }{t^{n-2\alpha }}\frac{\mathrm{d} t}{t}+\int_0^{\infty}\frac{\int_{B(x,t)}(\ia\sigma)^{{\gamma_2}q_1}\mathrm{d}\sigma }{t^{n-2\alpha }}\frac{\mathrm{d} t}{t}) + C\,\ia\sigma(x)\right]\nonumber\\
    &\leq {\lambda}^{q_{1}}\left[c(I + II)+ C\,\ia\sigma(x)\right]
\end{align}
We estimate I and II separately. Notice that lemma \ref{strcon} implies that $$I \leq c\, \ia\sigma(x).$$ Before estimating $II$, consider the following estimate
\begin{align*}
    \int_{B(x,t)}(\ia\sigma)^{{\gamma_2} q_1}\,\mathrm{d}\sigma&= \int_{B(x,t)}\Big[\int_0^\infty\Big(\frac{\sigma(B(y,r))}{r^{n-2\alpha }}\Big)\,\frac{\mathrm{d} r}{r}\Big]^{{\gamma_2} q_1}\,\mathrm{d}\sigma(y)\\
    & \leq c \int_{B(x,t)}\Big[\int_0^t\Big(\frac{\sigma(B(y,r))}{r^{n-\alpha p}}\Big)\,\frac{\mathrm{d} r}{r}\Big]^{{\gamma_2} q_1}\,\mathrm{d}\sigma(y)\\
    &\quad + c\int_{B(x,t)}\Big[\int_t^\infty\Big(\frac{\sigma(B(y,r))}{r^{n-\alpha p}}\Big)\,\frac{\mathrm{d} r}{r}\Big]^{{\gamma_2} q_1}\,\mathrm{d}\sigma(y) =: c\left(II_1+II_2\right),
\end{align*}
For $y\in B(x,t)$ and $r\leq t$, notice that $B(y,r)\subset B(x,2t)$, consequently
\begin{align*}
    II_1&= \int_{B(x,t)}\Big[\int_0^t\Big(\frac{\sigma(B(y,r))}{r^{n-\alpha p}}\Big)\,\frac{\mathrm{d} r}{r}\Big]^{{\gamma_2} q_1}\,\mathrm{d}\sigma(y)\\
    &\leq\int_{B(x,2t)}\Big[\int_0^t\Big(\frac{\sigma(B(y,r)\cap B(x,2t))}{r^{n-2\alpha }}\Big)\,\frac{\mathrm{d} r}{r}\Big]^{{\gamma_2} q_1}\,\mathrm{d}\sigma(y) \\
    & \leq\int_{B(x,2t)}\Big[\ia\sigma_{B(x,2t)}\Big]^{{\gamma_2} q_1}\,\mathrm{d}\sigma(y).
\end{align*}
Using \eqref{con10} again, we deduce
\begin{equation*}
    II_1\leq \int_{B(x,2t)}\left[\ia\sigma_{B(x,2t)}\right]^{{\gamma_2} q_1}\,\mathrm{d}\sigma(y)\leq c\,\sigma(B(x,2t)).
\end{equation*}
If $y\in B(x,t)$ and $r\geq t$ then $B(y,r)\subset B(x,2r)$, thus
\begin{align*}
    II_2&\leq \int_{B(x,t)}\Big[\int_t^\infty\Big(\frac{\sigma(B(x,2r))}{r^{n-2\alpha }}\Big)\,\frac{\mathrm{d} r}{r}\Big]^{{\gamma_2} q_1}\,\mathrm{d}\sigma(y) \\
    & = \sigma(B(x,t))\Big[\int_t^\infty\Big(\frac{\sigma(B(x,2r))}{r^{n-2\alpha}}\Big)\,\frac{\mathrm{d} r}{r}\Big]^{{\gamma_2} q_1}\\
    &\leq \sigma(B(x,t))\Big[\int_0^\infty\Big(\frac{\sigma(B(x,2r))}{r^{n-2\alpha }}\Big)\,\frac{\mathrm{d} r}{r}\Big]^{{\gamma_2} q_1}\leq c\,\sigma(B(x,t))\left[\ia\sigma (x)\right]^{{\gamma_2} q_1}.
\end{align*}
Combining $II_{1}$ and $II_{2}$ we obtain
\begin{equation*}
   \int_{B(x,t)}(\ia\sigma)^{{\gamma_2} q_1}\,\mathrm{d}\sigma\leq c\,\left[\sigma(B(x,2t))+\left(\ia\sigma (x)\right)^{{\gamma_2} q_1}\sigma(B(x,t))\right], 
\end{equation*}
and we finally conclude that $$II\leq c(\ia\sigma(x) +(\ia\sigma)^{\gamma1}),$$ so 
\begin{align*}
\ia(\overline{v}^{q_1}\mathrm{d}\sigma + d\mu_1)(x)&\leq \lambda^{q_{1}}(c(\ia\sigma(x) +(\ia\sigma)^{\gamma1}+C\ia\sigma(x))\\
&\leq \lambda^{q_{1}}c\,(\ia\sigma(x) +(\ia\sigma)^{\gamma1})
\end{align*}
In conclusion, by choosing $\lambda$ sufficiently large we have $$\ia(\overline{v}^{q_1}\mathrm{d}\sigma + d\mu_1)(x)\leq\overline{u}.$$

Similarly, the exact same reasoning for $\overline{u},q_{2},\gamma_{2},\mu_{2}$ instead of $\overline{v},q_{1},\gamma_{1},\mu_{1}$ should give that we can find $\lambda$ such that$$\ia(\overline{u}^{q_2}\mathrm{d}\sigma + d\mu_2)(x)\leq\overline{v}.$$ We conclude that $(\overline{u},\overline{v})$ is a supersolution. By arguing as in the proof of theorem \ref{con1}, we arrive at a solution $(u,v)$ satisfying 
\begin{align*}
    & \underline{u}\leq u\leq \overline{u},\\
    &  \underline{v}\leq v\leq \overline{v},
\end{align*}
in particular \eqref{estimateDoO} holds, which concludes the proof of theorem 2.
\section{Some Open problems}
We propose the following problems that could possibly generalize the ideas discussed here. 
\begin{enumerate}
\item Is it still possible to use the method of sub-super solutions to solve the more general system below?
\begin{equation}
     \left\{
\begin{aligned}
& (-\Delta)^{\alpha} u= \sigma\, u^{q_1}+ \mu\, v^{r_1}, \quad v>0 \quad \mbox{in}\quad \mathbb{R}^n,\\ 
& (-\Delta)^{\alpha} v= \sigma\, v^{q_2} + \mu\, v^{r_2}, \quad u>0 \quad \mbox{in}\quad \mathbb{R}^n,\\
& \liminf\limits_{|x|\to \infty}u(x)=0, \quad \liminf\limits_{|x|\to \infty}v(x)=0.
\end{aligned}
\right.
 \end{equation}
 \item Let $\sigma$ is radially symmetric and consider the system
 \begin{equation}
     \left\{
\begin{aligned}
& -\Delta_{p}u= \sigma\, v^{q_1}, \quad v>0 \quad \mbox{in}\quad \mathbb{R}^n,\\ 
& -\Delta_{p} v= \sigma\, u^{q_2}, \quad u>0 \quad \mbox{in}\quad \mathbb{R}^n,\\
& \liminf\limits_{|x|\to \infty}u(x)=0, \quad \liminf\limits_{|x|\to \infty}v(x)=0.
\end{aligned}
\right.
 \end{equation}
 What would be an equivalent condition to \eqref{limc} in this case?
\end{enumerate}
\bibliography{sn-article}
\end{document}